\documentclass[reqno]{amsart}

\usepackage{color}

\usepackage{cite}

\newtheorem{thm}{Theorem}[section]
\newtheorem{lem}[thm]{Lemma}
\newtheorem{cor}[thm]{Corollary}
\newtheorem{prop}[thm]{Proposition}

\newtheorem{rem}[thm]{Remark}

\numberwithin{equation}{section}

\newcommand{\Wqb}{W_{q,\mathcal{B}}}
\newcommand{\R}{\mathbb{R}}
\newcommand{\C}{\mathbb{C}}
\newcommand{\D}{\mathbb{D}}
\newcommand{\mS}{\mathbb{S}}

\newcommand{\N}{\mathbb{N}}
\newcommand{\A}{\mathbb{A}}
\newcommand{\B}{\mathbb{B}}
\newcommand{\T}{\mathbb{T}}

\newcommand{\E}{\mathbb{E}}

\newcommand{\ml}{\mathcal{L}}

\newcommand{\Om}{\Omega}

\newcommand{\ve}{\varepsilon}
\newcommand{\rd}{\mathrm{d}}
\newcommand{\dom}{\mathrm{dom}}

\newcommand{\dhr}{\mathrel{\lhook\joinrel\relbar\kern-.8ex\joinrel\lhook\joinrel\rightarrow}}

\begin{document}


\title[Age-Structured Diffusive Populations]{Properties of the Semigroup in $L_1$ Associated with Age-Structured Diffusive Populations}

\author{Christoph Walker}
\email{walker@ifam.uni-hannover.de}
\address{Leibniz Universit\"at Hannover\\ Institut f\" ur Angewandte Mathematik \\ Welfengarten 1 \\ D--30167 Hannover\\ Germany}
\date{\today}

\begin{abstract}
The linear semigroup associated with age-structured diffusive populations is investigated in the $L_1$-setting. A complete determination of its generator is given along with detailed spectral information that imply, in particular, an asynchronous exponential growth of the semigroup. Moreover, regularizing effects inherited from the diffusion part are exploited to derive additional properties of the semigroup.
\end{abstract}


\maketitle

\section{Introduction and Main Results}

A prototype model for the evolution of a diffusive population structured by 
age reads
\begin{subequations}\label{Eu1a}
\begin{align}
\partial_t u+\partial_a u&=\mathrm{div}_x\big(d(a,x)\nabla_xu\big)-m(a,x)u\ , && t>0\, , &  a\in (0,a_m)\, ,& & x\in\Om\, ,\label{u1a}\\
u(t,0,x)&=\int_0^{a_m} b(a,x)u(t,a,x)\,\rd a\, ,& & t>0\, , & & & x\in\Om\, ,\label{u2a}\\
\partial_\nu u(t,a,x)&=0\ ,& & t>0\, , &  a\in (0,a_m)\, ,& & x\in\partial\Om\, ,\label{u3a}\\
u(0,a,x)&=\phi(a,x)\ ,& & &  a\in (0,a_m)\, , & & x\in\Om\,.\label{u4a}
\end{align}
\end{subequations}
Here, $u=u(t,a,x)\ge 0$ is the population density at time $t\ge 0$, age $a\in [0,a_m)$ with maximal age $a_m\in (0,\infty]$, and spatial position $x\in\Om\subset \R^n$. The total population at a fixed time $t$ is  
$$
\int_0^{a_m} \int_\Omega u(t,a,x)\,\rd x\,\rd a\,.
$$
The age specific processes include death and birth processes with rates $m=m(a,x)\ge 0$  respectively $b=b(a,x)\ge 0$. Spatial dispersal is governed by the diffusion term in \eqref{u1a} with speed $d(a,x)>0$. The initial distribution of the population is $\phi=\phi(a,x)\ge 0$, and $\nu$ denotes the outward unit normal on $\partial \Om$. The investigation of linear and non-linear age-structured populations without and with spatial diffusion has a long history and there are many variants of Problem~\eqref{Eu1a} and different techniques to tackle them. We refer to \cite{Langlais88,Rhandi,RhandiSchnaubelt_DCDS99,WebbBook,WebbSpringer} and, for more recent contributions, to \cite{DM19,DM21,KangRuanJMB21,KangRuanJDE21,WalkerCrelle,WalkerDCDSA10,WalkerJEPE} and the references therein, though these lists are far from being complete.

Problem~\eqref{Eu1a} can be put in a more abstract framework by setting
$$
A(a)w:=\mathrm{div}_x\big(d(a,\cdot)\nabla_xw\big)-m(a,\cdot)w\, ,\quad w\in E_1\, ,
$$
where e.g. $E_1:=\Wqb^2(\Om)$  consists of all  functions $w$ in the Sobolev space  $W_q^{2}(\Om)$ with $q\in (1,\infty)$ satisfying the boundary condition $\partial_\nu w=0$ on $\partial\Om$. For a smooth and positive function $d$, the operator $A(a)$ is then the generator of an analytic semigroup in $E_0:=L_q(\Omega)$ with domain $E_1$.\\

Our attention is thus focused in the following on the abstract problem
\begin{subequations}\label{P} 
\begin{align}
\partial_t u+ \partial_au \, &=     A(a)u \,, \qquad t>0\, ,\quad a\in (0,a_m)\, ,\label{1}\\ 
u(t,0)&=\int_0^{a_m}b(a)\, u(t,a)\, \rd a\,, \qquad t>0\, ,\label{2} \\
u(0,a)&=  \phi(a)\,, \qquad a\in (0,a_m)\,,
\end{align}
\end{subequations}
for a function $u=u(t,a):\R^+\times [0,a_m)\rightarrow E_0^+$, where $a_m\in (0,\infty]$ and
$$
A(a): E_1\subset E_0\rightarrow E_0
$$ 
is for each $a\in [0,a_m)$ the generator of an analytic semigroup on some Banach lattice $E_0$ with domain $E_1$.  We shall be more specific about the assumptions when presenting the main results below. It is worth pointing out that the parabolic operator $A(a)$ and the age derivative $\partial_a$ --~being supplemented with a nonlocal boundary condition~\eqref{2}~-- act on different ``variables''.

It is known \cite{WebbSpringer} that a strongly continuous semigroup $(\mS(t))_{t\ge 0}$ in $\E_0:=L_1((0,a_m),E_0)$ can be associated with \eqref{P} if $A$ is independent of age and generates itself a strongly continuous semigroup on $E_0$. Indeed, integrating \eqref{1} formally  along characteristics gives the semigroup $(\mS(t))_{t\ge 0}$ almost explicitly (see \eqref{100} below). However, the corresponding infinitesimal generator~$\A$ has not been characterized completely except for the case that more restrictive conditions on the operator $A$ are imposed. More precisely, in \cite{WalkerMOFM} the generator~$\A$ is identified assuming the operator $A$ to possess {\it maximal $L_p$-regularity} (e.g. see \cite{LQPP} for more information on this property) restricting the phase space to  $L_p((0,a_m),E_0)$ with $p\in (1,\infty)$ and thus excluding the biologically ``natural'' space  $L_1((0,a_m),E_0)$ which is sufficient to give a meaning to the (local) overall population $\int_0^{a_m} u(t,a)\,\rd a$. The first aim of this research is to remedy this deficiency and improve the results of~\cite{WalkerMOFM}: we characterize the domain of the infinitesimal generator~$\A$ of the semigroup $(\mS(t))_{t\ge 0}$ also in the framework of  $\E_0=L_1((0,a_m),E_0)$ and without assuming the operator $A$ to have maximal $L_p$-regularity. The characterization of the generator in turn yields detailed information on its spectrum which implies, in particular, asynchronous exponential growth of the semigroup. 

The second aim of this research is then to provide further properties of the linear semigroup and its generator exploiting the regularizing properties inherited from the parabolic character of the diffusion operator. Such regularizing effects are derived from the explicit formula~\eqref{100} for the semigroup associated with \eqref{P}. They pave the way for the well-posedness \cite{WalkerEJAM,WalkerDCDSA10,WalkerJEPE} of nonlinear variants of \eqref{P} featuring a nonlinear operator $A=A(u)$ or a nonlinear birth rate $b=b(u)$.  Furthermore, the characterization of the generator in the framework of $\E_0=L_1((0,a_m),E_0)$ is particularly useful in the study of stability properties of equilibria in nonlinear problems, e.g. in order to derive a principle of linearized stability in a forthcoming work.

It is worth mentioning that other approaches than the one we choose herein (originating from \cite{WebbSpringer}), e.g. relying on integrated semigroups \cite{ThiemeDCDS,DM21}
or on perturbation techniques of Miyadera type  \cite{Rhandi,RhandiSchnaubelt_DCDS99} have been pursued as well and also yield the above mentioned asynchronous exponential growth. We also refer to the recent works \cite{KangRuanJMB21,KangRuanJDE21} on related age-structured equations with nonlocal diffusion.


\subsection*{Assumptions and Notations}

Set $J:=[0,a_m]$ if $a_m<\infty$ and $J:=[0,\infty)$ if $a_m=\infty$. We let $E_0$ be a real Banach lattice ordered by closed convex cone $E_0^+$ and 
assume throughout that 
\begin{equation*}
E_1\stackrel{d}{\dhr} E_0\,,
\end{equation*}
that is, $E_1$ is a dense subspace of $E_0$ with continuous and compact embedding. Given $\theta\in (0,1)$ and an admissible interpolation functor $(\cdot,\cdot)_\theta$ (see \cite{LQPP}),  we set $E_\theta:= (E_0,E_1)_\theta$. Then 
\begin{equation*}
E_1\stackrel{d}{\dhr} E_{\theta_1}\stackrel{d}{\dhr} E_{\theta_0} \stackrel{d}{\dhr}   E_0\,,\quad 0\le \theta_0<\theta_1\le 1\,.
\end{equation*} 
$E_\theta$ is equipped with the order naturally induced by $E_0^+$. We assume that there is $\rho>0$ such that
\begin{subequations}\label{A1}
\begin{equation}
A\in  C^\rho\big(J,\mathcal{H}(E_1,E_0)\big)\,,
\end{equation}
where $\mathcal{H}(E_1,E_0)$ is the subset of $\ml(E_1,E_0)$ of all generators of analytic semigroups on $E_0$ with domain $E_1$. 
Then \eqref{A1} and \cite[II.Corollary 4.4.2]{LQPP} imply that $A$ generates a parabolic evolution operator 
$$
\{\Pi(a,\sigma)\in\ml(E_0)\,;\, a\in J\,,\, 0\le\sigma\le a\}\,,
$$ 
on $E_0$ with regularity subspace $E_1$ in the sense of \cite[Section~II.2.1]{LQPP}. That is,
$$
\Pi\in C\big(J_\Delta,\ml_s(E_0)\big)\cap C\big(J_\Delta^*,\ml(E_0,E_1)\big)
$$
with $$
J_\Delta:= \{(a,\sigma)\in J\times J\,;\,  0\le\sigma\le a\}\,,\quad J_\Delta^*:= \{(a,\sigma)\in J\times J\,;\,  0\le\sigma< a\}
$$
satisfies
$$
\Pi(a,a)=1_{E_0}\,,\qquad \Pi(a,\sigma)=\Pi(a,s)\Pi(s,\sigma)\,,\quad (a,s), (s,\sigma)\in J_\Delta\,,
$$
and, for $a\in J$,
$$
\Pi(\cdot,a)\in C^1\big((a,a_m),\ml(E_0)\big) \,,\quad \Pi(a,\cdot)\in C^1\big([0,a),\ml_s(E_1,E_0)\big) 
$$
with
$$
\partial_1 \Pi(a,\sigma)=A(a)\Pi(a,\sigma)\,,\quad \partial_2 \Pi(a,\sigma)=- \Pi(a,\sigma) A(\sigma)\,,\quad (a,\sigma)\in J_\Delta^*\,.
$$
We further assume that
	there is $\varpi\in\R$ such that, if $\alpha \in [0,1]$, then 
	\begin{equation}\label{EO}
	\|\Pi(a,\sigma)\|_{\ml(E_\alpha)}+(a-\sigma)^\alpha\,\|\Pi(a,\sigma)\|_{\ml(E_0,E_\alpha)}\le M_\alpha e^{\varpi (a-\sigma)}\,,\qquad a\in J\,,\quad 0\le \sigma\le a \,,
	\end{equation}
	for some $M_\alpha\ge 1$ (if $a_m<\infty$ this is automatically satisfied, see~\cite[II.Lemma 5.1.3]{LQPP}) and
	\begin{equation}\label{A4}
	\text{if $a_m=\infty$, then $\varpi<0$ in \eqref{EO}}\,.
	\end{equation}
\end{subequations}
As for the birth rate $b$ we  assume that there is $\vartheta\in (0,1)$ with
\begin{equation}\label{A2}
b\in L_{\infty}\big(J,\ml(E_\theta)\big)\cap L_{1}\big(J,\ml(E_\theta)\big)\,, \quad \theta\in  \{0,\vartheta\}  
\,,
\end{equation}
and
\begin{equation}\label{A3}
b(a)\Pi(a,0)\in \ml_+(E_0) \ \text{is strongly positive}\footnote{Recall that if $E$ is an ordered Banach space, then $T\in \ml(E)$ is {\it strongly positive} if $Tz\in E$ is a {\it quasi-interior point} for each $z\in E^+\setminus\{0\}$, that is, if $\langle z',Tz\rangle_{E} >0$ for every $z'\in (E')^+\setminus\{0\}$.} \text{ for $a$ in a subset of $J$ of positive measure}\, .
\end{equation}
We set
$$
\|b\|_\theta:=\|b\|_{L_{\infty}(J,\ml(E_\theta))}
$$
in the following.

\begin{rem}
The assumptions that we impose are natural and easily checked in concrete applications. Indeed, \eqref{A1} and \eqref{A3} hold for uniformly elliptic operators satisfying the maximum principle while \eqref{A2} is a regularity assumption, e.g. see \cite{DanersKochMedina,WalkerJEPE}. For instance, the assumptions are satisfied for problem~\eqref{Eu1a} with $a_m<\infty$ provided that $E_0:=L_q(\Omega)$ and $E_1:=\Wqb^2(\Om)$  with $q\in (1,\infty)$, $d\in C^1(J\times\bar\Om)$  with $d(a,x)>0$, $m\in C^1(J, C(\bar\Om))$, and $b\in C(J, C^2(\bar\Om))$ is nonnegative and nontrivial (the regularities assumptions on $d, m$, and $b$ can even be weaken), see e.g. \cite{WalkerDCDSA10,WalkerJEPE}.
\end{rem}

Note that if $\lambda\in\C$, then
$$
\Pi_\lambda(a,\sigma):=e^{-\lambda (a-\sigma)}\Pi(a,\sigma)\,,\qquad a\in J\,,\quad 0\le\sigma\le a\,,
$$
is the evolution operator associated with $-\lambda+A(a)$. 
 In particular, for $x\in E_0$ and \mbox{$\phi\in \E_0=L_1(J,E_0)$}, the function $v\in C(J,E_0)$, given by
\begin{equation}\label{VdK}
v(a)=\Pi_\lambda(a,0)x+\int_0^a\Pi_\lambda(a,\sigma)\,\phi(\sigma)\,\rd \sigma\,,\quad a\in J\,,
\end{equation} 
is the {\it mild solution} to the Cauchy problem
$$
\partial_a v=(-\lambda+A(a))v+ \phi(a)\,,\quad a\in \dot{J}:=(0,a_m]\,,\qquad v(0)= x\,.
$$ 
Also recall from \cite[II. Theorems 1.2.1 \& 1.2.2]{LQPP} that 
\begin{equation}\label{strong}
\begin{split}
&\text{if $x\in E_1$ and $\phi\in C^\theta(J,E_0)+C(J,E_\theta)$ with $\theta \in (0,1]$,}\\
&\text{then $v\in  C^1(J,E_0)\cap C(J,E_1)$ is a strong solution\,.}
\end{split}
\end{equation}
We will use these facts frequently later on.

\subsection*{The Semigroup $(\mS(t))_{t\ge 0}$}

Integrating \eqref{1} formally along characteristics  yields that the solution 
$$
[\mS(t)\phi](a):=u(t,a)\,,\qquad t\ge 0\,,\quad a\in J\,,
$$ 
to \eqref{P}  with $\phi\in\E_0=L_1(J,E_0)$ is given by
\begin{subequations}\label{100}
  \begin{equation}
     \big[\mS(t) \phi\big](a)\, :=\, \left\{ \begin{aligned}
    &\Pi(a,a-t)\, \phi(a-t)\, ,& &   a\in J\,,\ 0\le t\le a\, ,\\
    & \Pi(a,0)\, B_\phi(t-a)\, ,& &  a\in J\, ,\ t>a\, ,
    \end{aligned}
   \right.
    \end{equation}
where $B_\phi:=u(\cdot,0)$ satisfies the Volterra equation 
    \begin{equation}\label{500}
    B_\phi(t)\, =\, \int_0^t \chi(a)\, b(a)\, \Pi(a,0)\, B_\phi(t-a)\, \rd
    a\, +\, \int_0^{a_m-t} \chi(a+t)\, b(a+t)\, \Pi(a+t,a)\, \phi(a)\, \rd a\, ,\quad
    t\ge 0\, ,
    \end{equation}
\end{subequations}		
with $\chi$ denoting the characteristic function of the interval $(0,a_m)$. That is, $B_\phi$ is such that
  \begin{equation}\label{6a}
    B_\phi(t)= \int_0^{a_m} b(a) \big[\mS(t)\phi\big](a)\, \rd a\, ,\quad t\ge 0\, .
    \end{equation}
The first result entails that $(\mS(t))_{t\ge 0}$ is a strongly continuous positive semigroup on $\E_0$ enjoying compactness properties and exhibiting regularizing effects induced by the parabolic evolution operator $\Pi$. We also provide a perturbation result in preparation for a future study of stability properties in nonlinear variants of~\eqref{P}.

\begin{thm}\label{T1}
Suppose \eqref{A1} and \eqref{A2}. \vspace{2mm}

{\bf (a)} $(\mS(t))_{t\ge 0}$ defined in \eqref{100} is a strongly continuous semigroup on $\E_0=L_1(J,E_0)$ which is eventually compact if $a_m<\infty$ and quasi-compact if $a_m=\infty$. If $b\in L_{\infty}\big(J,\ml_+(E_0)\big)$, then $(\mS(t))_{t\ge 0}$ is positive. \vspace{2mm} 

{\bf (b)} If $\alpha\in [0,1)$ and $\E_\alpha:=L_1(J,E_\alpha)$, then
\begin{equation}\label{E3}
\|\mS(t)\|_{\ml(\E_0,\E_\alpha)}\le M_\alpha\, e^{\varpi t}\,\left(\frac{\gamma\big(\|b\|_0 M_0 t,1-\alpha\big)}{(\|b\|_0 M_0)^{\alpha}}\, e^{\|b\|_0 M_0 t}+t^{-\alpha}\right)\,,\quad t> 0\,.
\end{equation}
In fact, 
\begin{equation}\label{E2}
\|\mS(t)\|_{\ml(\E_\theta)}\le M_\theta e^{(\varpi +\|b\|_\theta M_\theta)t}\,,\quad t\ge 0\,,\quad \theta\in  \{0,\vartheta\}  \,.
\end{equation}

{\bf (c)} Let $\A$ be the infinitesimal generator of the semigroup~$(\mS(t))_{t\ge 0}$. Consider $\alpha\in [0,1)$ and $\B\in\ml (\E_\alpha,\E_0)$. Then $\A+\B$ with $\dom(\A+\B):=\dom(\A)$ generates a strongly continuous semigroup~$(\T(t))_{t\ge 0}$ on $\E_0$ satisfying
\begin{equation}\label{E33T}
\T(t)\phi=\mS(t)\phi+\int_0^t\mS(t-s)\,\B\, \T(s)\,\phi\,\rd s\,,\quad t\ge 0\,,\quad \phi\in\E_0\,.
\end{equation}
 Moreover, there are $N_\alpha\ge 1$ and $\varsigma_\alpha\in\R$ such that
\begin{equation}\label{E3T}
\|\T(t)\|_{\ml(\E_0,\E_\alpha)}\le N_\alpha\, e^{\varsigma_\alpha t} t^{-\alpha}\,,\quad t> 0\,.
\end{equation}
If $b\in L_{\infty}\big(J,\ml_+(E_0)\big)$ and $\B\phi\in\E_0^+$ for $\phi\in\E_\alpha^+$, then the semigroup~$(\T(t))_{t\ge 0}$ is positive. 
\end{thm}

That $(\mS(t))_{t\ge 0}$ defines a strongly continuous positive semigroup on $\E_0=L_1(J,E_0)$ with the stated compactness properties can be verified by direct computations as in \cite[Theorem 4]{WebbSpringer} and \cite{WalkerMOFM}. The additional estimates~\eqref{E3} and~\eqref{E2} are due to \eqref{EO}, where $\gamma$ denotes the lower incomplete gamma function
$$
\gamma(x,\xi):=\int_0^x s^{\xi-1}\, e^{-s}\,\rd s\le \Gamma(\xi)\,,\qquad x\ge 0\,,\quad \xi>0 \,.
$$
Part~{\bf (c)} of Theorem~\ref{T1} relies on estimate~\eqref{E3}.\\

Motivated by \eqref{E2} we note that the restriction of $(\mS(t))_{t\ge 0}$ to $\E_\alpha$ defines a strongly continuous positive semigroup which is also a useful tool for the investigation of nonlinear problems.

\begin{cor}\label{T1B}
Suppose \eqref{A1} and \eqref{A2}. Given $\alpha\in [0,1)$, let $\mS_\alpha(t):=\mS(t)\vert_{\E_\alpha}$ for $t\ge 0$ be the restriction of $\mS(t)$ to $\E_\alpha=L_1(J,E_\alpha)$. Then $(\mS_\alpha(t))_{t\ge 0}$ is a strongly continuous  semigroup on $\E_\alpha$.  Restricting to $\alpha\in [0,\vartheta)$, the semigroup  $(\mS_\alpha(t))_{t\ge 0}$ is eventually compact if $a_m<\infty$ and quasi-compact if $a_m=\infty$.
\end{cor}

\subsection*{The Generator $\A$}

As pointed out in the introduction the infinitesimal generator $\A$ associated with the semigroup $(\mS(t))_{t\ge 0}$ is identified \cite{WalkerMOFM} only in $L_p(J,E_0)$ with $p\in (1,\infty)$ when assuming that $A$  has the property of maximal $L_p$-regularity. That this additional assumption is not needed for a characterization of $\A$ for the case $p=1$ is shown in the next theorem. It relies on an explicit formula for the resolvent of $\A$ (see~\eqref{inv} below). Setting
$$
\D:=\left\{\psi\in C^1(J,E_0)\cap  C(J,E_1)\,;\, \psi(0)=\int_0^{a_m} b(a) \psi(a)\,\rd a\right\}
$$
we also show that the subspace $\D$ is a core for the domain $D(\A)$ if $a_m<\infty$; that is, $\D$ is dense in $\dom(\A)$ when the latter is equipped with its graph norm.

\begin{thm}\label{T2}
Suppose \eqref{A1}, \eqref{A2}, and \eqref{A3}. Let $\A$ denote the infinitesimal generator of the semigroup~$(\mS(t))_{t\ge 0}$.\vspace{2mm}

{\bf (a)} $\psi\in \dom(\A)$ if and only if there exists $\phi\in \E_0$ such that $\psi\in C(J,E_0) \cap \E_0$ is the mild solution to
\begin{equation}\label{psi}
\partial_a\psi = A(a)\psi +\phi(a)\,,\quad a\in J\,,
\end{equation}
with 
\begin{equation}\label{psi0}
\psi(0)=\int_0^{a_m} b(a) \psi(a)\,\rd a\,.
\end{equation}
In this case, $\A \psi = -\phi$.\vspace{2mm}

{\bf (b)} The embedding $D(\A) \hookrightarrow \E_\alpha$ is continuous and dense for  $\alpha\in[0,1)$.\vspace{2mm}

{\bf (c)} If $a_m<\infty$  and if \eqref{A2} is valid also for $\theta=1$, then $\D$
is a core for $D(\A)$. If~$\psi\in \D$, then $\A\psi =-\partial_a\psi +A\psi$.
\end{thm}

That the domain of $\A$ is characterized in Theorem~\ref{T2}~{\bf (a)} in terms of {\it mild} solutions to~\eqref{psi} reflects the hyperbolic part of the operator $-\partial_a +A(a)$ while the regularizing effects stated in Theorem~\ref{T2}~{\bf (b)} are due to its parabolic part. 

For Theorem~\ref{T1} and Theorem~\ref{T2} it suffices that $E_0$ is an ordered Banach space; that is, no lattice property is needed.

\subsection*{Asynchronous Exponential Growth} 

Based on the compactness properties of the semigroup $(\mS(t))_{t\ge 0}$, the characterization of the generator $\A$ from Theorem~\ref{T2} entails information on its spectrum. We shall see that the spectrum is a pure point spectrum and in fact, if $\lambda\in\C$ (with $\mathrm{Re}\, \lambda>\varpi$ if $a_m=\infty$) is an eigenvalue of $\A$ with eigenvector  $\phi\in \mathrm{dom}(\A)\setminus\{0\}$, that is, if \mbox{$(\lambda-\A)\phi=0$}, then 
\begin{equation}\label{kerA}
\phi(a)=\Pi_\lambda(a,0)\phi(0)\,,\quad a\in J\,, \qquad \phi(0)=Q_\lambda \phi(0)\, ,
\end{equation} 
where the operator $Q_\lambda\in\ml(E_0)$ is defined as
\begin{equation}\label{Qlambda}
Q_\lambda :=\int_0^{a_m} b(a)\,  \Pi_\lambda(a,0)\,  \rd a\,.
\end{equation} 
Clearly, \eqref{kerA} implies that $1$ is an eigenvalue of~$Q_\lambda$ with eigenvector $\phi(0)$. The properties of the evolution operator, the compact embeddings of the interpolation spaces, and \eqref{A3}
 entail that $Q_\lambda$ is a compact and (for $\lambda\in\R$) strongly positive operator on  $E_0$. Hence, by the Krein-Rutman Theorem, the spectral radius $r(Q_\lambda)$ is positive and a simple eigenvalue of~$Q_\lambda$. Moreover, 
there is a unique $\lambda_0\in\R$ such that
\begin{equation}\label{lambda0}
r(Q_{\lambda_0})=1
\end{equation}
and there are a quasi-interior point $\zeta_{\lambda_0}$ in $E_0^+$ and a positive functional $\zeta_{\lambda_0}\in E_0'$ with $Q_{\lambda_0}\zeta_{\lambda_0}=\zeta_{\lambda_0}$ respectively  $Q_{\lambda_0}'\zeta_{\lambda_0}'=\zeta_{\lambda_0}'$. Actually, we shall see that $\lambda_0$ is a dominant and simple eigenvalue of~$\A$. This ensures asynchronous exponential growth of the semigroup~$(\mS(t))_{t\ge 0}$ in $\E_0=L_1(J,E_0)$ as stated in the next theorem.

\begin{thm}\label{T3}
Suppose \eqref{A1}, \eqref{A3}, and suppose \eqref{A2}  for  $\theta\in \{0,\vartheta,1\}$. Moreover, if $a_m=\infty$ suppose \mbox{$r(Q_0)\ge 1$}. Let $\lambda_0\in\R$ be as in \eqref{lambda0}. There are $\ve>0$ and~$N\ge 1$ such that
$$
\left\|e^{-\lambda_0 t}\, \mS(t)- P_{\lambda_0}\right\|_{\ml(\E_0)}\le N e^{-\ve t}\,,\quad t\ge 0\,.
$$
Here, $P_{\lambda_0}\in\ml(\E_0)$ is the spectral projection onto $\ker(\lambda_0-\A)$ given by
\begin{equation}\label{PP}
P_{\lambda_0}\phi=\frac{\langle \zeta_{\lambda_0}',H_{\lambda_0}\phi\rangle}{\langle \zeta_{\lambda_0}',\int_0^{a_m}  b(a)\Pi_{\lambda_0}(a,0)\,\zeta_{\lambda_0}\,a\,\rd a\rangle}\Pi_{\lambda_0}(\cdot,0)\zeta_{\lambda_0}\,,\quad \phi\in \E_0\,,
\end{equation}
where 
$$
 H_{\lambda_0}\phi:=\int_0^{a_m}b (a)\, \int_0^a\Pi_{\lambda_0}(a,\sigma)\,\phi(\sigma)\ \rd\sigma\,\rd a\in E_0\,.
 $$
\end{thm}

If $a_m<\infty$, then $\lambda_0\le\varpi+\|b\|_0 M_0$ and $\lambda_0$ coincides with the growth bound of the semigroup (and with the spectral bound of the generator $\A$), so the estimate \eqref{E2} for $\alpha=0$ can be improved to
\begin{equation*}
\|\mS(t)\|_{\ml(\E_0)}\le N e^{\lambda_0 t}\,,\quad t\ge 0\,,
\end{equation*}
for some $N\ge 1$.\\

As mentioned in the introduction, we shall use Theorem~\ref{T2} in a forthcoming research  to investigate stability properties of equilibria for nonlinear variants of~\eqref{P}. Regarding the linear problem~\eqref{P} an immediate consequence of Theorem~\ref{T3} (and Lemma~\ref{L0} below) are stability properties of the trivial equilibrium in terms of $r(Q_0)$.

\begin{cor}\label{stable}
Suppose \eqref{A1}, \eqref{A3}, and suppose \eqref{A2}  for  $\theta\in \{0,\vartheta,1\}$. Moreover, if $a_m=\infty$ suppose that \mbox{$r(Q_0)\ge 1$}.
\begin{itemize}
\item[{\bf (i)}]  If $r(Q_{0})<1$, then the zero equilibrium to \eqref{P} is globally exponentially asymptotically stable in $\E_0$. 
\item[{\bf (ii)}] If $r(Q_{0})=1$, then  the zero equilibrium to \eqref{P} is stable. Moreover, the solution $u$ to  \eqref{P} with $\phi\in \E_0$ converges exponentially toward an equilibrium.
\item[{\bf (iii)}] If $r(Q_{0})>1$, then  the zero equilibrium to \eqref{P} is unstable. More precisely, the solution~$u$ to  \eqref{P} with $\phi\in \E_0$ is asymptotic to the stable age distribution $e^{\lambda_0 t} P_{\lambda_0}\phi $ with $\lambda_0>0$ satisfying \eqref{lambda0} and $P_{\lambda_0}\phi$ being given by \eqref{PP}.
\end{itemize}
\end{cor}

Theorem~\ref{T3} can also be used to investigate the asynchronous exponential growth for semilinear equations. Indeed, consider
\begin{subequations}\label{P01} 
\begin{align}
\partial_t u+ \partial_au \, &=     A(a)u -m(u,a)u \,, \qquad t>0\, ,\quad a\in (0,a_m)\, ,\label{1a}\\ 
u(t,0)&=\int_0^{a_m}b(a)\, u(t,a)\, \rd a\,, \qquad t>0\, , \\
u(0,a)&=  \phi(a)\,, \qquad a\in (0,a_m)\,, 
\end{align}
\end{subequations}
with a semilinear term on the right-hand side of \eqref{1a} (representing a nonlinear death process) and suppose that the function $m=m(u,a)$ satisfies 
\begin{equation}\label{B1}
\begin{split} 
&m: \E_0\rightarrow L_\infty\big(J,\mathcal{L}(E_0)\big)\,,\ u\mapsto m(u,\cdot)\ \text{is uniformly Lipschitz continuous} \\
&\text{on bounded sets and}\ \|m(u,\cdot)\|_{L_\infty(J,\mathcal{L}(E_0))}\le f(\|u\|_{\E_0}) \text{ for } u\in \E_0\,,
\end{split}
\end{equation}
with a function $f$ such that
\begin{equation}\label{B2}
\begin{split} 
&f :\R^+\rightarrow\R^+ \ \text{ is non-increasing and }\ \displaystyle\int_{r_0}^\infty \frac{f(r)}{r}\,\rd r <\infty \text{ for } r_0>0\,.
\end{split}
\end{equation}
Then, given $\phi\in \E_0$, there is a unique  mild solution $u\in C(\R^+,\E_0)$ to \eqref{P01}. We introduce the nonlinear semigroup $\mathbb{U}$ by setting $\mathbb{U}(t)\phi:=u(t)$ and put
\begin{equation}\label{pp}
\mathbb{P}_{\lambda_0}(\phi):= P_{\lambda_0}\left( \phi+\int_0^\infty e^{-\lambda_0 s} F\big(\mathbb{U}(s)\phi\big)\,\rd s \right)\,,
\end{equation}
where $F(v):=-m(v,\cdot)v$. Then \cite[Theorem 1.1, Theorem 1.3]{GyllenbergWebb} implies:

\begin{cor}\label{C14}
Suppose \eqref{A1}, \eqref{A2} for  $\theta\in \{0,\vartheta,1\}$, \eqref{A3},   \eqref{B1}, and \eqref{B2}. Let $r(Q_0)>1$ so that  $\lambda_0>0$ in~\eqref{lambda0}.  There are $\ve>0$ and $N\ge 1$ such that
$$
\|e^{-\lambda_0 t}\, \mathbb{U}(t)\phi- \mathbb{P}_{\lambda_0}(\phi)\|_{\ml(\E_0)}\le N e^{-\ve t}\|\phi\|_{\E_0}\,,\quad t\ge 0\,,
$$
where  $\mathbb{U}(\cdot)\phi\in C(\R^+,\E_0)$ denotes the mild solution to~\eqref{P01} with $\phi\in\E_0$ and $\mathbb{P}_{\lambda_0}$ is defined in~\eqref{pp}.
\end{cor}

\subsection*{Outline}  Section~\ref{Sec3} is dedicated to the proof of the properties of the semigroup $(\mS(t)_{t\ge 0}$ as stated in Theorem~\ref{T1} and Corollary~\ref{T1B}.  The characterization of its generator $\A$ as stated in Theorem~\ref{T2} is provided in Section~\ref{Sec4}. It relies on the explicit formula~\eqref{inv} for the resolvent derived from its representation as the Laplace transform of the semigroup. Using the precise characterization of~$\A$ and the compactness property of $(\mS(t))_{t\ge 0}$, we then investigate in Section~\ref{Sec5} the spectrum of~$\A$ and show, in particular, that~$\lambda_0$ is a dominant and simple eigenvalue of $\A$. This, in turn, implies Theorem~\ref{T3} as well as Corollary~\ref{stable} and Corollary~\ref{C14}.  Finally, in the Appendix~\ref{Sec6} we sketch the proof of the existence of $B_\phi$ occurring in \eqref{500}.


\section{The Semigroup $(\mS(t))_{t\ge 0}$: Proofs of Theorem~\ref{T1} and Corollary~\ref{T1B}}\label{Sec3}

Suppose \eqref{A1} and \eqref{A2}.  As mentioned in the previous section, part {\bf (a)} of Theorem~\ref{T1} is mostly known. Indeed, it can be shown \cite[Lemma 2.1]{WalkerMOFM}  (see Lemma~\ref{L333} below) that there exists a mapping 
\begin{equation}\label{BB}
[\phi\mapsto B_\phi]\in\ml \big(\E_0,
C(\R^+,E_0)\big)
\end{equation}
such that $B_\phi$ is the unique solution to~\eqref{500}, and if $b\in L_{\infty}\big(J,\ml_+(E_0)\big)$  and $\phi\in \E_0^+$, then $B_\phi(t)\in E_0^+$ for $t\ge 0$. Based on this the proof that $(\mS(t))_{t\ge 0}$ defines a strongly continuous  positive  semigroup on $\E_0=L_1(J,E_0)$ is the same as in \cite[Theorem 4]{WebbSpringer} (proved for the case that $A$ is independent of age) to which we refer. 

\subsection*{Proof of Estimates \eqref{E3} and \eqref{E2}}

As for \eqref{E3} let $\phi\in \E_0$ and $\alpha\in [0,1)$. We note from \eqref{500}, \eqref{A2},  and~\eqref{EO}  that, for $\theta\in\{0,\vartheta\}$,  
$$
e^{-\varpi t}\| B_\phi(t)\|_{E_\theta}\le \|b\|_\theta\, M_\theta\,\int_0^te^{-\varpi a}\| B_\phi(a)\|_{E_\theta}\,\rd a +\|b\|_\theta\, M_\theta\, t^{-\theta}\,\|\phi\|_{\E_0}\,,\quad t>0\,,
$$
so that the singular Gronwall inequality~\cite[II.~Corollary 3.3.2]{LQPP} implies the existence of $c_\theta>0$ (with $c_0=1$) such that 
\begin{equation}\label{p1}
e^{-\varpi t}\| B_\phi(t)\|_{E_\theta}\le c_\theta\,\|b\|_\theta\, M_\theta\,t^{-\theta}\,\|\phi\|_{\E_0} e^{(1+\theta) \|b\|_\theta M_\theta t}  \,,\quad t>0\,,\quad \theta\in\{0,\vartheta\}\,.
\end{equation}
Thus, it follows from \eqref{100}, \eqref{EO}, and \eqref{p1} (with $\theta=0$) that
\begin{align*}
\|\mS(t)\phi\|_{\E_\alpha}&\le \int_0^t \|\Pi(a,0)\|_{\ml(E_0,E_\alpha)}\,\|B_\phi(t-a)\|_{E_0}\,\rd a\\
&\qquad +\int_t^{a_m} \|\Pi(a,a-t)\|_{\ml(E_0,E_\alpha)}\,\|\phi(a-t)\|_{E_0}\,\rd a\\
&\le  M_\alpha\,\|b\|_0\,M_0\,\|\phi\|_{\E_0} e^{(\varpi+\|b\|_0 M_0)t}\int_0^t e^{-\|b\|_0 M_0 a}\, a^{-\alpha}\,\rd a + M_\alpha\,e^{\varpi t}\, t^{-\alpha}\,\|\phi\|_{\E_0}\\
&=  
M_\alpha\, e^{\varpi t}\,\left(\frac{\gamma\big(\|b\|_0 M_0 t,1-\alpha\big)}{(\|b\|_0 M_0)^{\alpha}}\, e^{\|b\|_0 M_0 t}+t^{-\alpha}\right)\,\|\phi\|_{\E_0}\,,
\end{align*}
where we used implicitly that $t\le a_m$ for splitting the integral though the final estimate remains true for $t>a_m$, of course. This is \eqref{E3}. Since
$$
\gamma(x,1)=1-e^{-x}\,,\quad x\ge 0\,,
$$
we also obtain~\eqref{E2} for $\alpha=0$. The case $\theta=\vartheta$ of~\eqref{E2} one shows analogously by replacing~\eqref{p1} with the estimate
$$
e^{-\varpi t}\| B_\phi(t)\|_{E_\vartheta}\le \|b\|_\vartheta\, M_\vartheta\,\|\phi\|_{\E_\vartheta} e^{\|b\|_\vartheta M_\vartheta t}  \,,\quad t\ge 0\,,
$$
which also follows from Gronwall's inequality. \qed

\subsection*{Proof of Eventual Compactness when $a_m<\infty$} In order to prove that $\mS(t)$ is compact for $t>2a_m$,  we use Kolmogorov's compactness criterion \cite[Theorem A.1]{GutmanSIMA}. To this end, let $\mathcal{B}$ be a bounded subset of $\E_0$ and fix $t>2a_m$. Clearly, $\mS(t)\mathcal{B}$ is bounded in $\E_0$. Let $\phi\in\mathcal{B}$ and $h>0$. Note from \cite[II. Equation (5.3.8)]{LQPP} that, since $\vartheta\in (0,1)$, there is $c_1=c_1(a_m)>0$ with
$$
\| \Pi(a+h,0)-\Pi(a,0)\|_{\ml(E_{\vartheta},E_0)}\le c_1\, h^{\vartheta}\, ,\quad a+h\in J\,,
$$
and that \eqref{p1} implies the existence of $c_2=c_2(t,\mathcal{B})>0$ with
$$
\|B_\phi(t-a-h)\|_{E_{\vartheta}}\le c_2\, (t-a-h)^{-\vartheta}\,,\quad a+h\in J\,.
$$
Therefore, we infer from these observations along with \eqref{EO} and \eqref{100} that
\begin{equation*}
\begin{split}
\int_0^{a_m} \| (\widetilde{\mS(t)\phi})(a&+h)-(\mS(t)\phi)(a)\|_{E_0}\, \rd a \\
& \le  \int_0^{a_m-h} \| \Pi(a+h,0)-\Pi(a,0)\|_{\ml(E_{\vartheta},E_0)}\, \|B_\phi(t-a-h)\|_{E_{\vartheta}}\, \rd a \\ 
& \qquad + \int_0^{a_m-h} \|\Pi(a,0)\|_{\ml(E_0)}\, \|B_\phi(t-a-h)-B_\phi(t-a)\|_{E_0}\, \rd a\\ 
& \qquad + \int_{a_m-h}^{a_m} \|\Pi(a,0)\|_{\ml(E_0)}\, \|B_\phi(t-a)\|_{E_0}\, \rd a\\
&\le c_1 c_2 h^{\vartheta} \int_0^{a_m-h} (t-a-h)^{-{\vartheta}}\,\rd a\\
&\qquad+ M_0\, e^{\vert\varpi\vert a_m}\,\int_{t+h-a_m}^t \|B_\phi(s-h)-B_\phi(s)\|_{E_0}\,\rd s
\\
&\qquad+ M_0\, e^{\vert\varpi\vert a_m}\,\int_{t-a_m}^{t+h-a_m} \|B_\phi(s)\|_{E_0}\,\rd s
\end{split}
\end{equation*}
with tilde indicating the trivial extension.
As for the last two terms on the right-hand side of this estimate, by \eqref{BB} there is $c_3=c_3(t,\mathcal{B})>0$ with
$$
\|B_\phi(a)\|_{E_0}\le c_3\,,\quad a\in J\,,
$$
and we thus infer from \eqref{500} that, for $a_m+h<t+h-a_m\le s\le t$,
\begin{equation*}
\begin{split}
  \|&B_\phi(s-h)-   B_\phi(s)\|_{E_0}\\
	&\le \Big\|\int_{s-h-a_m}^{s-h} b(s-h-a)\, \Pi(s-h-a,0)B_\phi(a)\, \rd a -\int_{s-a_m}^s b(s-a)\, \Pi(s-a,0)B_\phi(a)\, \rd a \Big\|_{E_0}\\
	&\le \int_{s-h-a_m}^{s-h} \big\|b(s-h-a)\, \Pi(s-h-a,0) -b(s-a)\, \Pi(s-a,0)\big\|_{\ml(E_0)}\,\| B_\phi(a)\|_{E_0}\, \rd a\\
	&\qquad +\left(\int_{s-h}^{s}+\int_{s-h-a_m}^{s-a_m}\right)\big\| b(s-a)\, \Pi(s-a,0)\big\|_{\ml(E_0)}\,\| B_\phi(a)\|_{E_0}\, \rd a\\
&\le c_3\int_{s-h}^{s-h-a_m} \big\|b(s-h-a)\, \Pi(s-h-a,0) -b(s-a)\, \Pi(s-a,0)\big\|_{\ml(E_0)}\,\rd a\\
&\qquad  +c_3 \left(\int_{s-h}^{s}+\int_{s-h-a_m}^{s-a_m}\right)\big\| b(s-a)\, \Pi(s-a,0)\big\|_{\ml(E_0)}\, \rd a\,.
\end{split}
\end{equation*}
Noticing that \eqref{A2} and \eqref{EO} imply
$$
\big[ a\mapsto b(a)\Pi(a,0)\big]\in L_1(J,\ml(E_0))\,,
$$
we conclude that
\begin{equation}\label{G1}
\lim_{h\to 0} \int_0^{a_m} \| (\widetilde{\mS(t)\phi})(a+h)-(\mS(t)\phi)(a)\|_{E_0}\, \rd a=0\,,\ \text{ uniformly w.r.t. $\phi\in\mathcal{B}$}\,.
\end{equation}
 Next, \eqref{500} and \eqref{p1} (with $\theta=0$) entail that
$$
\|(\mS(t)\phi)(a)\|_{E_{\vartheta}}\le \|\Pi(a,0)\|_{\ml(E_0,E_{\vartheta})}\, \| B_\phi(t-a)\|_{E_0}\le  c(t,\mathcal{B})\, a^{-\vartheta}\,,\quad a\in (0,a_m)\,.
$$
Given $\varepsilon>0$ let $R_\varepsilon$ be the $E_0$-closure of the ball in $E_{\vartheta}$ centered at $0$ of radius $c(t,\mathcal{B})\varepsilon^{-{\vartheta}}$. Then $R_\varepsilon$ is compact in $E_0$ due to the compact embedding of $E_{\vartheta}$ in $E_0$ and
\begin{equation}\label{G2}
(\mS(t)\phi)(a)\in R_\varepsilon\, ,\quad a\in J\setminus[0,\ve]\, ,\quad \phi\in \mathcal{B}\, .
\end{equation}
Therefore, \cite[Theorem A.1]{GutmanSIMA} along with \eqref{G1} and \eqref{G2} imply that $\mS(t)B$ is relatively compact in~$\E_0$. This completes the proof of Theorem~\ref{T1}.\qed

\subsection*{Proof of Quasi-Compactness when $a_m=\infty$} Given $t\ge 0$ and $\phi\in\E_0$ define
$$
     \big[K(t) \phi\big](a)\, :=\, \left\{ \begin{aligned}
    &0\, ,& & 0\le t\le a<\infty\, ,\\
    & \Pi(a,0)\, B_\phi(t-a)\, ,& & 0\le a<t<\infty\, .
    \end{aligned}
   \right.
$$
Then
$$
\| \mS(t)\phi-K(t)\phi\|_{\E_0}=\int_t^\infty \|\Pi(a,a-t)\, \phi(a-t)\|_{E_0}\,\rd a\le M_0\, e^{\varpi t}\,\|\phi\|_{\E_0} \rightarrow 0
$$
as $t\to\infty$ according to \eqref{A4}. Moreover, it is easy to adapt the proof above (for the case $a_m<\infty$) to derive from Kolmogorov's compactness criterion \cite[Theorem A.1]{GutmanSIMA} that $K(t)\in\ml(\E_0)$ is a compact operator for each $t>0$. Thus, the semigroup $(\mS(t))_{t\ge 0}$ is quasi-compact (in the sense of \cite[V.~Definition~3.4]{EngelNagel}).\qed \vspace{2mm}

At this stage the proofs of parts {\bf (a)} and  {\bf (b)} of Theorem~\ref{T1} are complete and it only remains to prove part {\bf (c)}. In the following, $\A$ denotes the infinitesimal generator of the semigroup~$(\mS(t))_{t\ge 0}$.

\subsection*{Proof of Theorem~\ref{T1}~(c)}

Let $\alpha\in [0,1)$ and consider $\B\in\ml(\E_\alpha,\E_0)$. We shall see later in Corollary~\ref{C21} that $D(\A)$  is continuously embedded into $\E_\alpha$ so that $\A+\B$ with $\dom(\A+\B):=\dom(\A)$ is well-defined. Recall from \eqref{E3} that there are $c_1>0$ and $\omega_1>0$ such that
\begin{equation}\label{S31}
\|\mS(t)\|_{\ml(\E_0,\E_\alpha)}\le  c_1\, e^{\omega_1 t}\, t^{-\alpha}\,,\quad t>0\,.
\end{equation}
Hence, if $t_0\in (0,1)$ and $\phi\in \E_0$, then
\begin{equation*}
\begin{split}
\int_0^{t_0}\|\B\,\mS(t)\phi\|_{\E_0}\,\rd t &\le \int_0^{t_0}\|\B\|_{\ml(\E_\alpha,\E_0)}\,\|\mS(t)\|_{\ml(\E_0,\E_\alpha)}\,\|\phi\|_{\E_0}\,\rd t\\
&\le c_2\,\|\B\|_{\ml(\E_\alpha,\E_0)}\, t_0^{1-\alpha}\,\|\phi\|_{\E_0}\,.
\end{split}
\end{equation*}
Consequently, there are $t_0,q\in (0,1)$ such that
\begin{equation*}
\begin{split}
\int_0^{t_0}\|\B\,\mS(t)\phi\|_{\E_0}\,\rd t &\le q\,\|\phi\|_{\E_0}\,,\quad \phi\in\E_0\,.
\end{split}
\end{equation*}
We are thus in a position to apply the Miyadera-Voigt perturbation theorem~\cite[III.~Corollary 3.16]{EngelNagel} and conclude that $\A+\B$ generates a strongly continuous semigroup~$(\T(t))_{t\ge 0}$ on $\E_0$
satisfying 
\begin{equation}\label{E213}
\T(t)\phi=\mS(t)\phi+\int_0^t\mS(t-s)\,\B\, \T(s)\,\phi\,\rd s\,,\quad t\ge 0\,,\quad\phi\in D(\A)\,.
\end{equation}
If $\phi\in D(\A)$, then $\T(\cdot)\phi\in C(\R^+,D(\A))\hookrightarrow C(\R^+,\E_\alpha)$, so that \eqref{E213} and \eqref{S31} entail
$$
\|\T(t)\phi\|_{\E_\alpha}\le c_1\, e^{\omega_1 t}\, t^{-\alpha}\,\|\phi\|_{\E_0}+c_1\,\|\B\|_{\ml(\E_\alpha,\E_0)}\int_0^t e^{\omega_1 (t-s)}\, (t-s)^{-\alpha}\, \|\T(s)\phi\|_{\E_\alpha}\,\rd s\,,\quad t>0\,.
$$
The singular Gronwall inequality~\cite[II.~Corollary 3.3.2]{LQPP} now implies that there are $\N_\alpha\ge 1$ and $\varsigma_\alpha\in\R$ such that
$$
\|\T(t)\phi\|_{\E_\alpha}\le N_\alpha\, e^{\varsigma_\alpha t}\, t^{-\alpha}\,\|\phi\|_{\E_0}\,,\quad t> 0\,,\quad \phi\in D(\A)\,.
$$
The density of $D(\A)$ in $\E_0$ yields \eqref{E3T}. Moreover, in combination with \eqref{E3T}, \eqref{E213}, and Lebesgue's theorem, the density of $D(\A)$ in $\E_0$ also entails~\eqref{E33T}.

Suppose now that $\B\phi\in\E_0^+$ for $\phi\in\E_\alpha^+$ and take $\lambda>0$ sufficiently large. Then $\lambda-\A$ and $\lambda-\A-\B$ are invertible with
\begin{equation*}
(\lambda-\A-\B)^{-1}= (\lambda-\A)^{-1}\big(1-\B(\lambda-\A)^{-1}\big)^{-1}= (\lambda-\A)^{-1}\sum_{j=0}^\infty \big[\B(\lambda-\A)^{-1}\big]^j\,,
\end{equation*}
where the proof of \cite[III.~Theorem 3.14]{EngelNagel} shows that the Neumann series converges since $\B$ is a Miyadera-Voigt perturbation of $\A$. This formula together with the positivity of $\B$ and the fact that $\A$ is resolvent positive since the semigroup $(\mS(t))_{t\ge 0}$ is positive imply that $\A+\B$ is resolvent positive. Hence, the semigroup~$(\T(t))_{t\ge 0}$ is positive. This proves part~{\bf (c)} of Theorem~\ref{T1}.\qed

\medskip

\subsection*{Proof of Corollary~\ref{T1B}}

Let $\alpha\in [0,1)$ and recall from \eqref{E3} and $\E_\alpha\hookrightarrow\E_0$ that $\mS_\alpha(t)\in\ml(\E_\alpha)$ for $t\ge 0$, where $\mS_\alpha(t)$  is the restriction of $\mS(t)$ to $\E_\alpha=L_1(J,E_\alpha)$. Thus, in order to prove that $(\mS_\alpha(t))_{t\ge 0}$  is a strongly continuous positive semigroup on $\E_\alpha$ it suffices to prove the strong continuity. To this end, let $\phi\in\E_\alpha$.
We then obtain from  \eqref{100}  and \eqref{EO} for $t\in (0,a_m)$ that
\begin{equation*}
\begin{split}
\|\mS(t)\phi-\phi\|_{\E_\alpha}&\le \int_0^t \|\Pi(a,0)\|_{\ml(E_0,E_\alpha)}\,\| B_\phi(t-a)\|_{E_0}\,\rd a
+\int_0^t \|\Pi(a,0)\|_{\ml(E_\alpha)}\,\| \phi(a)\|_{E_\alpha}\,\rd a
\\
&\quad  +\int_t^{a_m}\|\Pi(a,a-t)\|_{\ml(E_\alpha)}\,\|\phi(a-t)-\phi(a)\|_{E_\alpha}\,\rd a\\
&\quad  +\int_t^{a_m}\|\Pi(a,a-t)\,\phi(a)-\phi(a)\|_{E_\alpha}\,\rd a\\
&\le M_\alpha \int_0^t e^{\varpi a}a^{-\alpha}\,\| B_\phi(t-a)\|_{E_0}\,\rd a
+M_\alpha \int_0^t e^{\varpi a}\,\|\phi(a)\|_{E_\alpha}\,\rd a \\
&\quad  +M_\alpha e^{\varpi t} \int_0^{a_m}\|\phi(a-t)-\phi(a)\|_{E_\alpha}\,\rd a\\
&\quad  +\int_t^{a_m}\|\Pi(a,a-t)\,\phi(a)-\phi(a)\|_{E_\alpha}\,\rd a\,.
\end{split}
\end{equation*}
As $t\to 0$, the first and the second term on the right-hand side converge to zero due to \eqref{p1} (with $\theta=0$) and $\phi\in\E_\alpha$,  the third term converges to zero since translations are strongly continuous on $\E_\alpha$, and the last term goes to zero due to Lebesgue's theorem.
This proves the strong continuity of $(\mS_\alpha(t))_{t\ge 0}$. That this semigroup in $\E_\alpha$ with $\alpha\in [0,\vartheta)$ is eventually compact if $a_m<\infty$ respectively quasi-compact if $a_m=\infty$ one shows as above (for the case $\E_0$) using the fact that $E_{\vartheta}$ embeds compactly in $E_\alpha$ for $0\le \alpha < \vartheta\le 1$. This yields Corollary~\ref{T1B}.\qed

\section{The Generator $\A$: Proof of Theorem~\ref{T2}}\label{Sec4}

We next turn to the identification of the generator $\A$ of the semigroup $(\mS(t)_{t\ge 0}$ which is crucial for what follows. To this end, suppose \eqref{A1}, \eqref{A2}, and \eqref{A3}.

\subsection*{Resolvent Representation Formula} 
In the following,
$$
I:=\R\ \text{ if }\ a_m<\infty\,,\qquad I:=(\varpi,\infty)\ \text{ if }\ a_m=\infty\,.
$$
Recall that the operators $Q_\lambda$ for $\lambda\in \C$ with $\mathrm{Re}\,\lambda\in I$ are defined in \eqref{Qlambda}. Their spectral radii determine the spectrum of $\A$ as we shall see below. Let us first note from  \eqref{EO} and \eqref{A2}  the regularizing property 
$$
Q_\lambda\in\mathcal{L}(E_0,E_\vartheta)
$$
and hence \mbox{$Q_\lambda\in\mathcal{L}(E_0)$} is compact due to the compact embedding of $E_\vartheta$ in $E_0$. Consequently, \mbox{$\sigma(Q_\lambda)\setminus\{0\}$} consists only of eigenvalues. Moreover, \eqref{A3} implies that $Q_\lambda \in\mathcal{L}(E_0)$ is strongly positive for $\lambda\in I$. Based on the Krein-Rutman Theorem, the following result is shown in  \cite[Lemma 2.4, Lemma 2.5]{WalkerMOFM}.

\begin{lem}\label{L0} 
For $\lambda\in I$,  the spectral radius $r(Q_\lambda)$ is positive and a simple eigenvalue of \mbox{$Q_\lambda\in\ml(E_0)$} with an eigenvector $\zeta_\lambda \in E_1$ that is quasi-interior in $E_0^+$. Moreover, $r(Q_\lambda)$ is an eigenvalue of the dual operator $Q_\lambda'\in\ml(E_0')$ with a positive eigenfunctional~$\zeta_\lambda'\in E_0'$.  
The mapping  
$$
I\rightarrow (0,\infty)\,,\quad \lambda\mapsto r(Q_\lambda)
$$ 
is continuous and strictly decreasing with 
$$
\lim_{\lambda\rightarrow\infty}r(Q_\lambda) =0\,.
$$
If $a_m<\infty$, then
$$
\lim_{\lambda\rightarrow-\infty}r(Q_\lambda) =\infty\,.
$$
\end{lem}

According to Lemma~\ref{L0}, if $\lambda \in I$ is large enough, then $r(Q_\lambda)<1$ so that   $(1-Q_\lambda)^{-1}\in\ml(E_0)$. 
Next, introducing for $\lambda\in I$ the operator $H_\lambda$ as
$$
 H_{\lambda}\phi:=\int_0^{a_m}b (a)\, \int_0^a\Pi_{\lambda}(a,\sigma)\,\phi(\sigma)\ \rd\sigma\,\rd a\,,\quad \phi\in \E_0\,,
 $$
we can state the following auxiliary result for further use.

\begin{lem}\label{AL1}
Let $\lambda \in I$. Then 
$$
H_\lambda\in\ml(\E_\theta,E_\theta)\cap\ml\big(L_\infty(J,E_0),E_\theta\big)\,,\quad \theta\in \{0,\vartheta\}\,.
$$
Moreover, if $a_m<\infty$, if \eqref{A2} is valid for $\theta=1$, and if $\phi\in C(J,E_\xi)+C^\xi(J,E_0)$ for some $\xi\in (0,1]$, then $H_\lambda\phi\in E_1$.
\end{lem}

\begin{proof}
Let $\theta\in \{0,\vartheta\}$. 
Noticing from \eqref{EO} that 
\begin{equation*}
\begin{split}
\left\|  H_\lambda\phi \right\|_{E_\theta}&\le \int_0^{a_m}\|b (a)\|_{\ml(E_\theta)} \int_0^a\|\Pi_{\lambda}(a,\sigma)\|_{\ml(E_0,E_\theta)}\,\|\phi(\sigma)\|_{E_0}\ \rd\sigma\,\rd a\\
&\le M_\theta  \int_0^{a_m}\|b (a)\|_{\ml(E_\theta)}\int_0^a e^{(\varpi-\lambda) (a-\sigma)} \, (a-\sigma)^{-\theta} \, \rd\sigma\,\rd a\, \|\phi\|_{L_\infty(J,E_0)} 
\end{split}
\end{equation*}
for $\phi\in L_\infty(J,E_0)$, it readily follows from \eqref{A2} that $H_\lambda\in \ml\big(L_\infty(J,E_0),E_\theta\big)$. Similarly,
\begin{equation*}
\begin{split}
\left\|  H_{\lambda}\phi \right\|_{E_\theta}&\le \int_0^{a_m}\|b (a)\|_{\ml(E_\theta} \int_0^a\|\Pi_{\lambda}(a,\sigma)\|_{\ml(E_\theta)}\,\|\phi(\sigma)\|_{E_\theta}\ \rd\sigma\,\rd a\\
&\le M_\theta  \int_0^{a_m}\|b (a)\|_{\ml(E_\theta)}\int_0^a e^{(\varpi-\lambda) (a-\sigma)} \,  \,\|\phi(\sigma)\|_{E_\theta}\, \rd\sigma\,\rd a
\end{split}
\end{equation*}
for $\phi\in \E_\theta$, 
so that again \eqref{A2} implies $H_\lambda\in \ml(\E_\theta,E_\theta)$.

 Finally, let $a_m<\infty$ and consider $\phi\in C(J,E_\xi)+C^\xi(J,E_0)$ for some $\xi\in (0,1]$. Setting
$$
v(a):=\int_0^a\Pi_{\lambda}(a,\sigma)\,\phi(\sigma)\ \rd\sigma\,,\quad a\in J\,,
$$
we have $v\in C(J,E_1)$ by \eqref{strong}. Hence, $H_\lambda\phi\in E_1$ provided \eqref{A2} is valid for $\theta=1$.
\end{proof}

The following representation formula for the resolvent of $\A$ is fundamental for determining the domain of $\A$. It has already been observed in \cite{WalkerMOFM} (but was then used only under more restrictive conditions). We include the proof here for the reader's ease. Recall that the growth bound of the semigroup $(\mS(t))_{t\ge 0}$ given by
$$
\omega_0:=\inf\{\omega\in\R\,;\, \exists\, M\ge 1: \|\mS(t)\|_{\ml(\E_0)}\le Me^{\omega t}\,,\, t\ge 0\}
$$
while 
$$
s(\A):=\sup\{\mathrm{Re}\,\lambda\,;\,\lambda\in\sigma(\A)\}
$$
is the spectral bound of the generator $\A$.
Setting
$$
\omega_*:=\varpi+ M_0 \|b\|_0  
$$
we have $\omega_*\ge \omega_0\ge s(\A)$ due to \eqref{E2}.  
 

\begin{prop}\label{P1}
If  $\lambda>\omega_*$ with $(1-Q_\lambda)^{-1}\in\ml(E_0)$, then
\begin{equation}\label{inv}
\big[(\lambda-\A)^{-1}\phi\big](a)=\int_0^a \Pi_\lambda(a,\sigma)\,\phi(\sigma)\, \rd \sigma +\Pi_\lambda(a,0)(1-Q_\lambda)^{-1} H_\lambda\phi
\end{equation}
for $a\in J$ and $\phi\in\E_0$.
\end{prop}

\begin{proof} The choice of $\lambda$ ensures that it belongs to the resolvent set of $\A$ and that $H_\lambda\in\ml(\E_0,E_0)$. Using the Laplace transform formula 
$$
(\lambda-\A)^{-1}\phi=\int_0^\infty e^{-\lambda t}\, \mS(t) \phi\, \rd t
$$
for $\phi\in\E_0$, we infer from \cite[p.69 f]{HillePhillips} and \eqref{100} that, for a.a. $a\in J$,
\begin{equation*}
\begin{split}
\big[(\lambda-\A)^{-1}\phi\big](a)&=\int_0^\infty e^{-\lambda t}\, \big[\mS(t) \phi\big](a)\, \rd t\\
&=\int_0^a \Pi_\lambda(a,t)\,\phi(t)\, \rd t +\Pi_\lambda(a,0) \int_0^\infty e^{-\lambda t} B_\phi(t)\, \rd t
\, .
\end{split}
\end{equation*}
Since $\lambda>\omega_*$, it follows from \eqref{p1} that
$$
\Psi:=\int_0^\infty e^{-\lambda t} B_\phi(t)\, \rd t\in E_0\,,
$$
and using \eqref{100} and \eqref{6a}, we obtain
\begin{equation*}
\begin{split}
\Psi&=\int_0^{a_m}b(a)\int_0^\infty e^{-\lambda t}\, \big[\mS(t) \phi\big](a)\, \rd t\,\rd a\\
&=\int_0^{a_m} b(a)\, \Pi_\lambda(a,0)\, \rd a\, \Psi +\int_0^{a_m} b(a)\int_0^a  \Pi_\lambda(a,t)\, \phi(t) \,\rd t\,\rd a= Q_\lambda\Psi +H_\lambda\phi\, ,
\end{split}
\end{equation*}
that is,
$$
 \Psi= (1-Q_\lambda)^{-1} H_\lambda\phi
$$
from which the claim follows.
\end{proof}

We obtain also the following information on the domain of $\A$.

\begin{cor}\label{C21}
If $\alpha\in[0,1)$, then the embedding $D(\A) \hookrightarrow \E_\alpha$ is continuous and dense.
\end{cor}

\begin{proof}
Recall from \eqref{E3} that there are $c_2>0$ and $\omega_2>0$ such that
\begin{equation}\label{cstar}
\|\mS(t)\|_{\ml(\E_0,\E_\alpha)}\le c_2\, e^{\omega_2 t}\, \left(t^{-\alpha}+1\right)\,,\quad t> 0\,.
\end{equation}
Fix $\lambda>\max\{\omega_0,\omega_2\}$. Given $\psi\in\dom(\A)$ set $\phi:=(\lambda-\A)\psi\in\E_0$. Since
$$
(\lambda-\A)^{-1}\phi=\int_0^\infty e^{-\lambda t}\, \mS(t) \phi\, \rd t\,,
$$
we derive from~\eqref{cstar} that
$$
\|\psi\|_{\E_\alpha}\le \int_0^\infty e^{-\lambda t}\, \|\mS(t)\|_{\ml(\E_0,\E_\alpha)}\,\|\phi\|_{\E_0}\,\rd t\le c_2\,\int_0^\infty e^{(-\lambda+\omega_2) t}\,\big(t^{-\alpha}+1\big)\,\rd t\,\|\phi\|_{\E_0}\le c_3\,\|\psi\|_{D(\A)}\,
$$
what yields the continuity of the embedding $D(\A) \hookrightarrow \E_\alpha$. That this embedding is also dense follows from the fact that, for $\phi\in\E_\alpha$,
$$
\phi_t:=\frac{1}{t}\int_0^t \mS(s)\phi\,\rd s\in D(\A) \,,\quad t>0\,,
$$
and
$$
\phi_t\to\phi \quad\text{ in }\  \E_\alpha\quad \text{ as }\ t\to 0
$$
since $(\mS(t))_{t\ge 0}$ is strongly continuous on $\E_\alpha$ by Corollary~\ref{T1B}.
\end{proof}

\subsection*{Proof of Theorem~\ref{T2}~(a)}

Consider $\psi\in\dom (\A)$ and fix $\lambda>\omega_*$ with $(1-Q_\lambda)^{-1}\in\ml(E_0)$. Then
\begin{equation*}
\phi_0:=(\lambda-\A)\psi\in\E_0\,,\qquad \psi=(\lambda-\A)^{-1}\phi_0\,,
\end{equation*}
and Proposition~\ref{P1} entails that
$$
\psi(a)=\int_0^a \Pi_\lambda(a,\sigma)\,\phi_0(\sigma)\, \rd \sigma +\Pi_\lambda(a,0)\psi(0)\,,\quad a\in J\,,
$$
with
$$
\psi(0)=(1-Q_\lambda)^{-1} H_\lambda\phi_0\in E_0\,.
$$
That is, $\psi\in C(J,E_0)$ is the mild solution to
$$
\partial_a\psi =(-\lambda+A(a))\psi+\phi_0(a)\,,\quad a\in J\,,
$$
and the computation in Proposition~\ref{P1} along with $\psi(0)=(1-Q_\lambda)^{-1} H_\lambda\phi_0$ imply that $\psi$ satisfies~\eqref{psi0}.
Setting $\phi:=\phi_0-\lambda\psi=-\A\psi\in\E_0$ we thus derive that $\psi\in C(J,E_0)$ is indeed the mild solution to~\eqref{psi}~-~\eqref{psi0} 
as claimed in Theorem~\ref{T2}~{\bf (a)}.

Conversely, suppose that $\psi\in \E_0\cap C(J,E_0)$ is the mild solution to
$$
\partial_a\psi =A(a)\psi+\phi(a)\,,\quad a\in J\,,
$$
for some $\phi\in\E_0$ and $\psi$ satisfies~\eqref{psi0}.
Taking $\lambda>\omega_0$  with   $(1-Q_\lambda)^{-1}\in\ml(E_0)$, we set 
$$
\phi_0:=\lambda\psi+\phi\in\E_0
$$ 
so that
$\psi$ is the mild solution to
$$
\partial_a\psi =(-\lambda+A(a))\psi+\phi_0(a)\,,\quad a\in J\,,
$$
given by
\begin{equation}\label{l1}
\psi(a)=\Pi_\lambda(a,0)\psi(0)+\int_0^a\Pi_\lambda(a,\sigma)\phi_0(\sigma)\,\rd \sigma\,,\quad a\in J\,.
\end{equation}
Therefore,
$$
\psi(0)=\int_0^{a_m} b(a)\,\psi(a)\,\rd a=Q_\lambda\psi(0)+H_\lambda\phi_0
$$
and 
\begin{equation}\label{l2}
\psi(0)=(1-Q_\lambda)^{-1} H_\lambda\phi_0\,.
\end{equation}
Proposition~\ref{P1} together with \eqref{l1} and \eqref{l2} imply that $\psi=(\lambda-\A)^{-1}\phi_0\in\dom(\A)$. This proves  part {\bf (a)} of Theorem~\ref{T2}.\qed \\

\subsection*{Proof of Theorem~\ref{T2}~(b)}
This is shown in Corollary~\ref{C21}.\qed\\


\subsection*{Proof of Theorem~\ref{T2}~(c)}
Let $a_m<\infty$ and let \eqref{A2} be valid for $\theta=1$. We  show that
$$
\D=\left\{\psi\in C^1(J,E_0)\cap C(J,E_1)\,;\, \psi(0)=\int_0^{a_m} b(a) \psi(a)\,\rd a\right\}
$$
is a core for $D(\A)$. To this end, let $\psi\in\D$ and set $\phi:=\partial_a\psi-A\psi\in C(J,E_0)\subset \E_0$. Obviously, $\psi$ is a strong solution to
$$
\partial_a\psi =A(a)\psi+\phi(a)\,,\quad a\in J\,,
$$
satisfying \eqref{psi0}.
Thus, from  Theorem~\ref{T2}~{\bf (a)} we deduce $\psi\in \dom(\A)$ with 
$$
\A\psi=-\phi=-\partial_a\psi+ A(\cdot)\psi\,.
$$ 
Therefore, $\D\subset \dom(\A)$. To prove that this inclusion is dense (with respect to the graph norm), let $\psi\in\dom(\A)$ and $\ve>0$ be arbitrary. Choosing $\theta\in (0,1)$ and  $\lambda>\omega_*$ with $(1-Q_\lambda)^{-1}\in\ml(E_0)$, we set $\phi:=(\lambda-\A)\psi\in \E_0$. Then, there is $\phi_\ve\in C(J,E_\theta)$ such that $\|\phi_\ve-\phi\|_{\E_0}\le \ve$. By part {\bf (a)} of Theorem~\ref{T2}, $\psi_\ve:=(\lambda-\A)^{-1}\phi_\ve\in\dom(\A)$ is the mild solution to
$$
\partial_a\psi_\ve =(-\lambda+A(a))\psi_\ve+\phi_\ve(a)\,,\quad a\in J\,,
$$
with (see \eqref{l2})
\begin{equation*}
\psi_\ve(0)=(1-Q_\lambda)^{-1} H_\lambda\phi_\ve
\end{equation*}
so that $\psi_\ve(0)\in E_1$ owing to Lemma~\ref{AL1} and since $(1-Q_\lambda)^{-1}\in\ml(E_1)$ due to \eqref{EO} and the assumption that \eqref{A2} is valid also for $\theta=1$.
Therefore, \eqref{strong} implies $\psi_\ve\in\D$. Moreover,
$$
\|\psi_\ve-\psi\|_{D(\A)}\le \left\| (\lambda-\A)^{-1}\right\|_{\ml(\E_0,D(\A))}\,\|\phi_\ve-\phi\|_{\E_0}
$$
showing that $\D$ is indeed dense in $D(\A)$ as claimed.\qed

\begin{rem}
Independent of whether $a_m<\infty$ or $a_m=\infty$, if $\psi\in W_1^1(J,E_0)\cap L_1(J,E_1)$ satisfies \eqref{psi0},
then $\psi\in\dom(\A)$ and $\A\psi=-\partial_a\psi+A\psi$.
\end{rem}

\begin{proof}
Set $\phi:=-\partial_a\psi+A\psi\in \E_0$ and note that the properties of the evolution operator \cite{LQPP} and the regularity $\psi\in W_1^1(J,E_0)\cap L_1(J,E_1)$ guarantee that,
$$
\frac{\partial}{\partial \sigma}\Pi(a,\sigma)\psi(\sigma)=\Pi(a,\sigma)\phi(\sigma)\,,\quad \text{a.a. } \sigma\in (0,a)\,,\quad a\in J\,.
$$
Integrating with respect to $\sigma$ yields that $\psi\in C(J,E_0)$  is a mild solution to \eqref{psi} satisfying \eqref{psi0}. Hence, $\psi\in\dom(\A)$ with $\A\psi=\phi=-\partial_a\psi+A\psi$ according to Theorem~\ref{T2}~(a).
\end{proof}

\section{Spectral Properties: 
Proof of Theorem~\ref{T3}}\label{Sec5}

The main ideas of the proof of Theorem~\ref{T3} are reminiscent of \cite{WalkerMOFM,WalkerJEPE}, but the details differ. We thus include a full proof herein for which we impose throughout assumptions \eqref{A1},  
\eqref{A3},  and assume \eqref{A2}  for  $\theta\in \{0,\vartheta,1\}$.
Moreover, we assume for this section that
\begin{equation}\label{assump2}
\text{if $a_m=\infty$, then  $r(Q_0)\ge 1$}\,.
\end{equation}
Note that
\begin{equation}\label{rp}
Q_\lambda\in\mathcal{L}(E_0,E_\vartheta)\cap \mathcal{L}(E_\vartheta,E_1)\cap \mathcal{L}(E_1)
\end{equation}
due to \eqref{EO} and since \eqref{A2} is valid for  $\theta\in \{0,\vartheta,1\}$.
Also note that \eqref{A4}, \eqref{assump2}, and Lemma~\ref{L0} imply that there is a unique $\lambda_0\in \R$  such that
$$
r(Q_{\lambda_0})=1\,
$$ 
and that $\lambda_0\ge 0$ if $a_m=\infty$.

\subsection*{Spectrum of $\A$} The compactness property of $(\mS(t))_{t\ge 0}$ stated in Theorem~\ref{T1} provides  information on the spectrum $\sigma(\A)$ of the generator~$\A$, in particular, that it is a pure and discrete point spectrum. Moreover, the eigenvalues~$\mu$  of~$\A$ are related to the operator $Q_\mu$. Recall that we have introduced the interval $I$ as
$$
I=\R\ \text{ if }\ a_m<\infty\,,\qquad I=(\varpi,\infty)\ \text{ if }\ a_m=\infty\,.
$$

\begin{lem}\label{L23}
{\bf (a)} If $a_m<\infty$, then the spectrum $\sigma(\A)$  is countable and consists of poles of the resolvent $R(\cdot,\A)$ of finite algebraic multiplicities (in particular, $\sigma(\A)$ is a pure point spectrum). Moreover, the set $\{\lambda\in \sigma(\A)\,;\,\mathrm{Re}\, \lambda\ge r\}$ is finite for each $r\in\R$.\vspace{2mm}

{\bf (b)} If $a_m=\infty$, then  the set $\{\lambda\in \sigma(\A)\,;\,\mathrm{Re}\, \lambda\ge 0\}$ is finite and consists of poles of the resolvent~$R(\cdot,\A)$ of finite algebraic multiplicities. \vspace{2mm}

{\bf (c)}  Let $\mu\in\C$ with $\mathrm{Re}\,\mu\in I$. Then $\psi\in \mathrm{ker}(\mu-\A)$ if and only if there is $\psi_0\in E_1$ with
\begin{equation}\label{psiQ}
\psi(a)=\Pi_\mu(a,0)\psi_0\,,\quad a\in J\,,\qquad \psi_0=Q_\mu\psi_0\,.
\end{equation}
In particular,  $\mathrm{ker}(\mu-\A)\subset\D$. \vspace{2mm}

{\bf (d)} Let $m\in\N$. Then, $\mu\in \sigma(\A)$ has geometric multiplicity $m$ if and only if $1\in \sigma(Q_\mu)$ has geometric multiplicity $m$.
\end{lem}

\begin{proof}
{\bf (a)} Since the semigroup $(\mS(t))_{t\ge 0}$ is eventually compact if \mbox{$a_m<\infty$} according to Theorem~\ref{T1}, this is a consequence of \cite[V. Corollary 3.2]{EngelNagel}.\vspace{2mm}

{\bf (b)} Since the semigroup $(\mS(t))_{t\ge 0}$ is quasi-compact if $a_m=\infty$ according to Theorem~\ref{T1}, this is a consequence of \cite[V.~Theorem~3.7]{EngelNagel}.\vspace{2mm}

{\bf (c)} Let $\mu\in \sigma(\A)$ and $\psi\in \mathrm{ker}(\mu-\A)\subset\dom(\A)$.  Since $\A\psi=\mu\psi$, we infer from Theorem~\ref{T2}~(a) that $\psi$ is the mild solution to
\begin{equation}\label{pop}
\partial_a\psi=A(a)\psi-\mu\psi(a)\,,\quad a\in J\,,\qquad \psi(0)=\int_0^{a_m} b(a)\, \psi(a)\,\rd a\,,
\end{equation}
with $\phi:=-\mu\psi\in C(J,E_0)$. Clearly, \eqref{pop} implies \eqref{psiQ} with $\psi_0=\psi(0)\in E_1$ due to \eqref{rp}. Hence,  $\psi\in \D$ by~\eqref{strong}.  
 Conversely, if $\psi$ satisfies \eqref{psiQ}, then $\psi$ obviously satisfies \eqref{pop}. Moreover, since $\mathrm{Re}\,\mu\in I$, we have $-\mu\psi\in\E_0$ due to \eqref{psiQ} and \eqref{EO}. Therefore,
$\psi\in \dom(\A)$  with $(\mu-\A)\psi=0$ owing to  Theorem~\ref{T2}~(a). This proves~{\bf (c)}.\vspace{2mm}

{\bf (d)} Let $\mu\in \sigma(\A)$ have geometric multiplicity $m\in\N$. Then, there are  linearly independent $\psi_1,\ldots,\psi_m\in\mathrm{ker}(\mu-\A)$, and \eqref{psiQ} yields
$$
\psi_j(a)=\Pi_\mu(a,0)\psi_j(0)\,,\quad a\in J\,,\qquad \psi_j(0)=Q_\mu\psi_j(0)\,.
$$
This readily implies that $\psi_1(0),\ldots,\psi_m(0)\in E_0$ are linearly independent eigenvectors of $Q_\mu$ corresponding to the eigenvalue $1$.\vspace{2mm}

Conversely, let $1\in \sigma(Q_\mu)$ have geometric multiplicity $m\in\N$ so that there are linearly independent $\Psi_1,\ldots,\Psi_m\in E_0$  with $\Psi_j=Q_\mu\Psi_j$. Set
$$
\psi_j(a):=\Pi_\mu(a,0)\Psi_j\,,\quad a\in J\,,\quad j=1,\ldots,m\,.
$$
Then  $\psi_j\in\mathrm{ker}(\mu-\A)$ due to {\bf (c)}. If $\zeta:=\sum_j\xi_j\psi_j=0$ for some $\xi_j\in\C$, the unique solvability of
$$
\partial_a\zeta=(-\mu+A(a))\zeta\,,\quad a\in J\,,\qquad \zeta(0)= \sum_j\xi\Psi_j
$$
readily implies $\zeta(0)=0$, hence $\xi_j=0$ so that that $\psi_1,\ldots,\psi_m$ are linearly independent.
\end{proof}

We next characterize the spectral bound $s(\A)$. 

\begin{prop}\label{P23}
$s(\A)=\lambda_0$ is a simple and dominant eigenvalue of $\A$ and 
$$
\ker(\lambda_0-\A)=\mathrm{span}\{\Pi_{\lambda_0}(\cdot,0)\zeta_{\lambda_0}\}
$$
with $\zeta_{\lambda_0}\in E_1$ from Lemma~\ref{L0}.
\end{prop}

\begin{proof}
Recall from \eqref{assump2} that $\lambda_0\ge 0$ if $a_m=\infty$. Set $s:=s(\A)$ and note from  \cite[Corollary~12.9]{BFR} that $s\in\sigma(\A)$ since $(\mS(t))_{t\ge 0}$ is a positive semigroup on the Banach lattice $\E_0$. Since $r(Q_{\lambda_0})=1$ is a simple eigenvalue of $Q_{\lambda_0}$ with eigenvector $\zeta_{\lambda_0}\in E_0^+$, it follows  from Lemma~\ref{L23} that 
$$
\ker(\lambda_0-\A)=\mathrm{span}\{\varphi\}\,,\qquad \varphi:=\Pi_{\lambda_0}(\cdot,0)\zeta_{\lambda_0}\,.
$$ 
In particular,  $\lambda_0\le s$. Owing to Lemma~\ref{L23} (and the fact that $s\ge \lambda_0\ge 0$ if $a_m=\infty$), the set
$$
\sigma_0:=\{\lambda\in\sigma(\A)\,;\, \mathrm{Re}\,\lambda=s\}
$$
has only finitely many elements, while  \cite[Theorem 8.14]{ClementEtAl} entails that it is additively cyclic since $s$ is a pole of the resolvent of $\A$ and $(\mS(t))_{t\ge 0}$ a positive semigroup. Consequently, $\sigma_0=\{s\}$ so that~$s$ is a dominant eigenvalue. Next note that Lemma~\ref{L23} implies $1\in\sigma(Q_s)$, hence $r(Q_s)\ge 1$ and thus $s\le \lambda_0$ by Lemma~\ref{L0}. Consequently, $s=\lambda_0$. 

It remains to prove that $s=\lambda_0$ is simple. To this end, consider $\psi\in \ker({\lambda_0}-\A)^2$. Then 
$$
\phi_0:=({\lambda_0}-\A)\psi\in\ker({\lambda_0}-\A)
$$ 
so that $\phi_0=\gamma\varphi$ for some $\gamma\in \C$. We may assume without loss of generality that $\gamma$ is real and  that  $\gamma>0$. Choose $\tau>0$ such that $\tau\zeta_{{\lambda_0}}+\psi(0)\in E_0^+$, define then
$$
p:=\tau\varphi+\psi\in\dom(\A)\,,
$$
and note that $({\lambda_0}-\A)p=\phi_0=\gamma\varphi$. 
Theorem~\ref{T2} now implies that $p$ is the mild solution to
$$
\partial_a p=(-{\lambda_0}+A(a))p +\gamma\varphi(a)\,,\quad a\in J\,,\qquad p(0)=\tau\zeta_{{\lambda_0}}+\psi(0)\in E_0^+\,.
$$
From \eqref{VdK} we derive that
\begin{equation*}
p(a)=\Pi_{\lambda_0}(a,0)p(0)+\gamma\int_0^a\Pi_{\lambda_0}(a,\sigma)\,\Pi_{\lambda_0}(\sigma,0)\,\zeta_{\lambda_0}\,\rd \sigma\,,\quad a\in J\,,
\end{equation*}
hence, using the property 
\begin{equation}\label{iiii}
\Pi_{\lambda_0}(a,\sigma)\Pi_{\lambda_0}(\sigma,0)=\Pi_{\lambda_0}(a,0)\,,\qquad a\in J\,,\quad 0\le \sigma\le a\,,
\end{equation}
we get
\begin{equation*}
p(a)=\Pi_{\lambda_0}(a,0)p(0)+\gamma\, a\,\Pi_{\lambda_0}(a,0)\,\zeta_{\lambda_0}\,,\quad a\in J\,.
\end{equation*}
Since
$$
p(0)=\int_0^{a_m} b(a)\,p(a)\,\rd a
$$
by Theorem~\ref{T2}, we thus infer that
$$
(1-Q_{\lambda_0})p(0)=\gamma\int_0^{a_m} b(a)\,\Pi_{\lambda_0}(a,0)\,\zeta_{\lambda_0}\,a\,\rd a \in E_0^+\,.
$$
However, since $r(Q_{\lambda_0})=1$ and $Q_{\lambda_0}$ is strongly positive, \cite[Corollary~12.4]{DanersKochMedina} ensures that this equation
has no positive solution $p(0)\in E_0^+$ if the right-hand side is non-trivial. Consequently,~\eqref{A3} entails $\gamma=0$ from which we deduce that $\phi_0=0$, hence $\ker({\lambda_0}-\A)^2=\ker({\lambda_0}-\A)$. Therefore, $\lambda_0=s$ is a simple eigenvalue of $\A$.
\end{proof}

\begin{cor}
If $a_m<\infty$, then $\omega_0=s(\A)=\lambda_0$. In particular, there is $N\ge 1$ such that
$$
\|\mS(t)\|_{\ml(\E_0)}\le Ne^{\lambda_0 t}\,,\quad t\ge 0\,.
$$
\end{cor}

\begin{proof}
Since $(\mS(t))_{t\ge 0}$ is eventually compact, we have $\omega_0=s(\A)$ by \cite[IV. Corollary 3.12]{EngelNagel}.
\end{proof}

Note from \eqref{E2} that $\lambda_0\le \varpi +\|b\|_0 M_0$ if $a_m<\infty$.\\

We are now in a position to prove the asynchronous exponential growth of~$(\mS(t))_{t\ge 0}$ and identify the corresponding projection as stated in Theorem~\ref{T3}.

\subsection*{Proof of Theorem~\ref{T3}}

Recall that $\lambda_0$ is a dominant and simple eigenvalue of $\A$ according to Proposition~\ref{P23} and that $(\mS(t))_{t\ge 0}$ is eventually compact if $a_m<\infty$ respectively quasi-compact if $a_m=\infty$ due to Theorem~\ref{T1}. It thus follows from \cite[V.~Corollary~3.3]{EngelNagel} respectively \cite[V.~Theorem~3.7]{EngelNagel} that there are $\ve>0$ and~$N\ge 1$ such that
\begin{equation*}
\left\|e^{-\lambda_0 t}\, \mS(t)- P_{\lambda_0}\right\|_{\ml(\E_0)}\le N e^{-\ve t}\,,\quad t\ge 0\,,
\end{equation*}
where 
\begin{equation}\label{Pl}
P_{\lambda_0}=\lim_{\lambda\to\lambda_0} (\lambda-\lambda_0)(\lambda-\A)^{-1}\in\ml(\E_0)
\end{equation}
is the spectral projection onto $\ker(\lambda_0-\A)=\mathrm{span}\{\Pi_{\lambda_0}(\cdot,0)\zeta_{\lambda_0}\}$ (see also \cite[IV. \S 1.17]{EngelNagel}). The identification of $P_{\lambda_0}$ is the same as in \cite{WalkerMOFM}, we sketch it here for the sake of completeness.

Consider $\phi\in\E_0$ and note from \eqref{Pl} and \eqref{inv} that
$$
P_{\lambda_0}\phi=\lim_{\lambda\to\lambda_0} (\lambda-\lambda_0)\Pi_\lambda(a,0)(1-Q_\lambda)^{-1} H_{\lambda_0}\phi\,.
$$
Since $E_0=\R \zeta_{\lambda_0}\oplus \mathrm{rg}(1-Q_{\lambda_0})$ by Lemma~\ref{L0}, we may write
\begin{equation*}
H_{\lambda_0}\phi=\langle \zeta_{\lambda_0}', H_{\lambda_0}\phi\rangle_{E_0}\,\zeta_{\lambda_0} + (1-Q_{\lambda_0})g\big(H_{\lambda_0}\phi\big)
\end{equation*}
for some $g(H_{\lambda_0}\phi)\in E_0$, 
where $\zeta_{\lambda_0}'\in E_0'$ is the positive eigenfunctional $\zeta_{\lambda_0}'\in E_0'$ of $Q_{\lambda_0}'$ from Lemma~\ref{L0} with $Q_{\lambda_0}'\zeta_{\lambda_0}=\zeta_{\lambda_0}'$ and normalization $\langle \zeta_{\lambda_0}', \zeta_{\lambda_0}\rangle_{E_0}=1$.
Since 
$$
1-Q_{\lambda_0}=1-Q_\lambda+Q_\lambda-Q_{\lambda_0}
$$  
implies
$$
\lim_{\lambda\to\lambda_0} (\lambda-\lambda_0)\,\Pi_\lambda(\cdot,a)\,(1-Q_\lambda)^{-1} (1-Q_{\lambda_0})g\big(H_{\lambda_0}\phi\big)=0\,,
$$
we thus infer
\begin{equation}\label{u1}
P_{\lambda_0}\phi=\langle \zeta_{\lambda_0}', H_{\lambda_0}\phi\rangle_{E_0}\, \lim_{\lambda\to\lambda_0} (\lambda-\lambda_0)\,\Pi_\lambda(\cdot,a)\,(1-Q_\lambda)^{-1} \zeta_{\lambda_0}\,.
\end{equation}
Hence, writing
\begin{equation}\label{u2}
P_{\lambda_0}\phi=c(\phi)\Pi_{\lambda_0}(\cdot,0)\zeta_{\lambda_0}
\end{equation}
with $c(\phi)\in\R$, we deduce from \eqref{u2} and \eqref{u1} that
\begin{equation*}
\begin{split}
c(\phi)\zeta_{\lambda_0}&= c(\phi)\, Q_{\lambda_0}\zeta_{\lambda_0}=\int_0^{a_m} b(a)\, (P_{\lambda_0}\phi)(a)\,\rd a\\
&= \langle \zeta_{\lambda_0}', H_{\lambda_0}\phi\rangle_{E_0}\, \lim_{\lambda\to\lambda_0} (\lambda-\lambda_0)\,Q_\lambda\,(1-Q_\lambda)^{-1} \zeta_{\lambda_0}\,.
\end{split}
\end{equation*}
Applying $\zeta_{\lambda_0}'\in E_0'$ of $Q_{\lambda_0}'$ on both sides yields
$$
 c(\phi)= c_1\, \langle \zeta_{\lambda_0}', H_{\lambda_0}\phi\rangle_{E_0}
$$
for some constant $c_1$ not depending on $\phi$ (but on $\zeta_{\lambda_0}$ and $\zeta_{\lambda_0}'$, of course). Consequently, from \eqref{u2},
\begin{equation*}
P_{\lambda_0}\phi=c_1\, \langle \zeta_{\lambda_0}', H_{\lambda_0}\phi\rangle_{E_0}\, \Pi_{\lambda_0}(\cdot,0)\zeta_{\lambda_0}\,,\quad \phi\in \E_0\,.
\end{equation*}
The constant $c_1$ is readily computed from the fact that $P_{\lambda_0}^2=P_{\lambda_0}$ and that (see~\eqref{iiii})
$$
H_{\lambda_0}\big( \Pi_{\lambda_0}(\cdot,0)\zeta_{\lambda_0}\big)=\int_0^{a_m} b(a)\,  \int_0^a \Pi_{\lambda_0}(a,\sigma)\, \Pi_{\lambda_0}(\sigma,0)\zeta_{\lambda_0}\, \rd \sigma\,\rd a=\int_0^{a_m} b(a)\,  \Pi_{\lambda_0}(a,0)\zeta_{\lambda_0}\, a\,\rd a
$$
to yield formula~\eqref{PP}. This completes the proof of Theorem~\ref{T3}.\qed  
 
\subsection*{Proof of Corollary~\ref{stable}} This is a consequence of Theorem~\ref{T3} and Lemma~\ref{L0}.\qed

\subsection*{Proof of Corollary~\ref{C14}} This follows from Theorem~\ref{T3} and \cite[Theorem 1.1, Theorem 1.3]{GyllenbergWebb}.\qed

\section{Appendix}\label{Sec6}

We verify the existence of $B_\phi$ occurring in \eqref{500}. More precisely, given a function
\begin{equation}\label{D1}
h\in C(\R^+,E_0)
\end{equation}
we show the solvability of
\begin{equation}\label{500oh}
\begin{split}
    B_\phi^h(t)\, &=\, \int_0^{t} \chi(a)\, b(a)\, \Pi(a,0)\, B_\phi^h(t-a)\, \rd a \\
    &\quad +\int_0^{a_m-t} \chi(a+t)\, b(a+t)\, \Pi(a+t,a)\, \phi(a)\, \rd a +h(t)  
	\end{split}
    \end{equation}
for $t\ge 0$.

\begin{lem}\label{L333}
Let $h$ satisfy \eqref{D1}. There is a mapping 
\begin{equation*}
[\phi\mapsto B_\phi^h]\in\ml \big(\E_0, C(\R^+,E_0)\big)
\end{equation*}
such that $B_\phi^h$ is the unique solution to~\eqref{500oh}.
\end{lem}

\begin{proof}
Define
$$
\big(\mathcal{K}B\big)(t):= \int_0^{t} \chi(a)\, b(a)\, \Pi(a,0)\, B(t-a)\, \rd a \,,\qquad t\ge 0\,,\quad B\in C(\R^+,E_0)\,.
$$
For fixed $\phi\in\E_0$, equation~\eqref{500oh} is equivalent to $B=B_\phi^h$ satisfying
\begin{equation}\label{b1}
(1-\mathcal{K})B=H_\phi^h
\end{equation}
in $C(\R^+,E_0)$, 
where $H_\phi^h\in C(\R^+,E_0)$ is defined as
$$
H_\phi^h(t):=\int_0^{a_m-t}\chi(a+t)\,b(a+t)\,\Pi(a+t,a)\,\phi(a)\,\rd a+h(t)\,,\quad t\ge 0\,.
$$
Let $T>0$ be arbitrary. We claim that  $\mathcal{K}\in \ml\big(C([0,T],E_0)\big)$ is compact.  Let $\mathbb{B}_0$ denote the unit ball in $C([0,T],E_0)$.
For $B\in \mathbb{B}_0$ and $t\in [0,T]$ we obtain, using~\eqref{EO} and \eqref{A4},
\begin{equation*}
\begin{split}
\|(\mathcal{K}B)(t)\|_{E_{\vartheta}}&\le  \int_0^{t} \chi(a)\, \|b(a)\|_{\ml(E_{\vartheta})}\, \|\Pi(a,0)\|_{\ml(E_0,E_{\vartheta})}\, \|B(t-a)\|_{E_0}\,\rd a\\
&\le \|b\|_{{\vartheta}}\, M_{\vartheta} \int_0^{t} a^{-{\vartheta}} e^{\varpi a}\, \rd a\, \|B\|_{C([0,T],E_0)}\le c(T)
\end{split}
\end{equation*}
so that $\mathcal{K}(\mathbb{B}_0)(t)$ is bounded in $E_{\vartheta}$ which compactly embeds into $E_0$. Therefore, $\mathcal{K}(\mathbb{B}_0)(t)$ is relatively compact in $E_0$ for each $t\in [0,T]$. Moreover, for $0\le s\le t\le T$ and $B\in \mathbb{B}_0$, we derive, using \eqref{A2} and \eqref{EO},
\begin{equation*}
\begin{split}
\|(\mathcal{K}B)(t)-(\mathcal{K}B)(s)\|_{E_0}&=\left\|\int_0^{t} \chi(t-a)\, b(t-a)\, \Pi(t-a,0)\, B(a)\, \rd a\right.\\
&\qquad \qquad \qquad \left. -\int_0^{s} \chi(s-a)\, b(s-a)\, \Pi(s-a,0)\, B(a)\, \rd a\right\|_{E_0}\\
& \le \int_0^{s}\left\|\chi(t-a)\, b(t-a)\, \Pi(t-a,0)\right.\\
&\qquad \qquad\qquad \qquad \  \left. -\chi(s-a)\, b(s-a)\, \Pi(s-a,0)\right\|_{\ml(E_0)}\,\rd a
\\
& \, \quad +\int_s^{t}\left\|\chi(t-a)\, b(t-a)\, \Pi(t-a,0)\right\|_{\ml(E_0)}\,\rd a\,.
\end{split}
\end{equation*}
If $b\in BC(J,\ml(E_0))$, then the equi-continuity follows from  $b(\cdot)\Pi(\cdot,0)\in C([r,a_m),\ml(E_0))$ with $r>0$ small enough \cite[II. Eq. (2.1.2)]{LQPP} and splitting the first integral into integrals on \mbox{$(0,s-r)$} and $(s-r,s)$. The general case follows by approximating $b\in L_\infty(J,\ml(E_0))$ pointwise a.e. by a sequence $(b_j)\subset BC(J,\ml(E_0))$ with $\|b_j\|_\infty\le \|b\|_\infty$, see \cite[p.133f]{LQPP}, and using Egorov's theorem (see \cite[Theorem~4.3, step (iii) of the proof]{SimonettWalker} for a similar argument).
This implies that $\mathcal{K}(\mathbb{B}_0)\subset C([0,T],E_0)$ is equi-continuous. Therefore, the Arzel\'a-Ascoli Theorem ensures that $\mathcal{K}(\mathbb{B}_0)$ is compact in $C([0,T],E_0)$, hence $\mathcal{K}\in \ml\big(C([0,T],E_0)\big)$ is compact. Next, suppose that $\mathcal{K}B=B$ for some $B\in C([0,T],E_0)$. Then, using \eqref{A2} and \eqref{EO}, we deduce that there is $C(T)>0$ such that
\begin{equation*}
\begin{split}
\|B(t)\|_{E_0}&=\|(\mathcal{K}B)(t)\|_{E_0}\le \int_0^{t} \chi(a)\, \|b(a)\|_{\ml(E_0)}\, \|\Pi(a,0)\|_{\ml(E_0)}\, \|B(t-a)\|_{E_0}\,\rd a\\
&\le   \|b\|_{0}\, t\, c(T,\varpi)\, \|B\|_{C([0,T],E_0)} \le t \ C(T)\, \|B\|_{C([0,T],E_0)}
\end{split}
\end{equation*}
for $t\in [0,T]$. Inductively, we derive that
$$
\|B(t)\|_{E_0}\le \frac{t^n}{n!} C(T)^n \|B\|_{C([0,T],E_0)}\,,\quad t\in [0,T]\,,\quad n\in\N\,,
$$
so that $B\equiv 0$. Consequently, $1-\mathcal{K}$ is an isomorphism on $C([0,T],E_0)$ due to the Fredholm Alternative, and \eqref{b1} has for each $\phi\in\E_0$ a unique solution $B_\phi^h\in C([0,T],E_0)$. Moreover, it is readily seen that 
$$
[\phi\mapsto B_\phi^h]\in\ml \big(\E_0, C([0,T],E_0)\big)\,.
$$
Since $T>0$ was arbitrary, this implies the assertion.
\end{proof}

\bibliographystyle{siam}
\bibliography{AgeDiff}

\end{document}